\documentclass[leqno,12pt]{article} 
\usepackage{amsmath, amssymb}
\usepackage{amsthm} 
\theoremstyle{plain} 
\newtheorem{theorem}{\indent\sc Theorem}[section]

\newtheorem{corollary}[theorem]{\indent\sc Corollary}
\newtheorem{proposition}[theorem]{\indent\sc Proposition}

\theoremstyle{definition} 
\newtheorem{definition}[theorem]{\indent\sc Definition}

\newtheorem{example}[theorem]{\indent\sc Example}

\title{$\eta$-Ricci solitons on para-Kenmotsu manifolds} 
\author{Adara M. Blaga}
\date{} 
\begin{document}

\maketitle

\markboth{{\small\it {\hspace{8cm} $\eta$-Ricci solitons on para-Kenmotsu manifolds}}}{\small\it{$\eta$-Ricci solitons on para-Kenmotsu manifolds
\hspace{8cm}}}


\begin{abstract}
In the context of paracontact geometry, $\eta$-Ricci solitons are considered on manifolds satisfying certain curvature conditions: $(\xi,\cdot)_{R}\cdot S=0$, $(\xi,\cdot)_{S}\cdot R=0$, \linebreak $(\xi,\cdot)_{W_2}\cdot S=0$ and $(\xi,\cdot)_{S}\cdot W_2=0$. We prove that on a para-Kenmotsu manifold $(M,\varphi,\xi,\eta,g)$, the existence of an $\eta$-Ricci soliton implies that $(M,g)$ is quasi-Einstein and if the Ricci curvature satisfies $(\xi,\cdot)_{R}\cdot S=0$, then $(M,g)$ is Einstein. Conversely, we give a sufficient condition for the existence of an $\eta$-Ricci soliton on a para-Kenmotsu manifold.
\end{abstract}

\section{Introduction}

Ricci solitons represent a natural generalization of Einstein metrics on a Riemannian manifold, being generalized fixed points of Hamilton's Ricci flow $\frac{\partial }{\partial t}g=-2S$ \cite{ham}. The evolution equation defining the Ricci flow is a kind of nonlinear diffusion equation, an analogue of the heat equation for metrics. Under the Ricci flow, a metric can be improved to evolve into a more canonical one by smoothing out its irregularities, depending on the Ricci curvature of the manifold: it will expand in the directions of negative Ricci curvature and shrink in the positive case. Ricci solitons have been studied in many contexts: on K\"{a}hler manifolds \cite{cho}, on contact and Lorentzian manifolds \cite{bag}, \cite{cal}, \cite{ing}, \cite{sha}, \cite{tri}, on Sasakian \cite{fu}, \cite{he}, $\alpha$-Sasakian \cite{ing} and $K$-contact manifolds \cite{sha}, on Kenmotsu \cite{ba}, \cite{na} and $f$-Kenmotsu manifolds \cite{cal} etc. In paracontact geometry, Ricci solitons firstly appeared in the paper of G. Calvaruso and D. Perrone \cite{calv}. Recently, C. L. Bejan and M. Crasmareanu studied Ricci solitons on $3$-dimensional normal paracontact manifolds \cite{cr}.

\medskip

A more general notion is that of \textit{$\eta$-Ricci soliton} introduced by J. T. Cho and M. Kimura \cite{ch}, which was treated by C. C\u alin and M. Crasmareanu on Hopf hypersurfaces in complex space forms \cite{cacr}.

\medskip

In the present paper we shall consider $\eta$-Ricci solitons in the context of paracontact geometry, precisely, on a para-Kenmotsu manifold which satisfies certain curvature properties: $(\xi,\cdot)_{R}\cdot S=0$, $(\xi,\cdot)_{S}\cdot R=0$, $(\xi,\cdot)_{W_2}\cdot S=0$ and $(\xi,\cdot)_{S}\cdot W_2=0$, respectively. Remark that in \cite{na} H. G. Nagaraja and C. R. Premalatha have obtained some results on Ricci solitons satisfying conditions of the following type: $(\xi,\cdot)_{R}\cdot \tilde{C}=0$, $(\xi,\cdot)_{P}\cdot \tilde{C}=0$, $(\xi,\cdot)_{H}\cdot S=0$, $(\xi,\cdot)_{\tilde{C}}\cdot S=0$ and in \cite{ba} C. S. Bagewadi, G. Ingalahalli and S. R. Ashoka treated the cases: $(\xi,\cdot)_{R}\cdot B=0$, $(\xi,\cdot)_{B}\cdot S=0$, $(\xi,\cdot)_{S}\cdot R=0$, $(\xi,\cdot)_{R}\cdot \bar{P}=0$ and $(\xi,\cdot)_{\bar{P}}\cdot S=0$.

\section{Para-Kenmotsu manifolds}

Let $M$ be a $(2n+1)$-dimensional smooth manifold, $\varphi $ a tensor field of $(1, 1)$-type,
$\xi $ a vector field, $\eta $ a $1$-form and $g$ a
pseudo-Riemannian metric on $M$ of signature $(n+1, n)$. We say that $(\varphi , \xi , \eta , g)$ is an \textit{almost paracontact metric structure} on $M$ if \cite{s:z}:
\begin{enumerate}
  \item $\varphi \xi=0$, $\eta \circ \varphi =0$,
  \item $\eta (\xi )=1$, ${\varphi }^2=I_{\mathfrak{X}(M)}-\eta \otimes \xi $,
  \item $\varphi $ induces on the $2n$-dimensional distribution $\mathcal{D}:=\ker \eta $ an almost paracomplex structure $P$ i.e. $P^2=I_{\mathfrak{X}(M)}$ and the eigensubbundles $\mathcal{D}^+$, $\mathcal{D}^-$, corresponding to the eigenvalues $1$, $-1$ of $P$ respectively, have equal dimension $n$; hence $\mathcal{D}=\mathcal{D}^{+}\oplus \mathcal{D}^{-}$,
  \item $g(\varphi \cdot, \varphi \cdot)=-g+\eta \otimes \eta $.
  \end{enumerate}
We call $(M,\varphi , \xi , \eta , g)$ \textit{almost paracontact metric manifold}, $\varphi$ \textit{the structural endomorphism}, $\xi$ \textit{the characteristic vector field} and $\eta$ \textit{the paracontact form}.
Examples of almost paracontact metric structures are given in \cite{i:vz} and \cite{d:o}.

From the definition it follows that $\eta $ is the $g$-dual of $\xi $:
\begin{equation}
\eta (X)=g(X, \xi ),
\end{equation}
$\xi $ is a unitary vector field:
\begin{equation}
g(\xi , \xi )=1
\end{equation}
and $\varphi $ is a $g$-skew-symmetric operator:
\begin{equation}
g(\varphi X, Y)=-g(X,\varphi Y).
\end{equation}

Remark that the canonical distribution $\mathcal{D}$ is ${\varphi }$-invariant since $\mathcal{D}=Im \varphi $.
Moreover, $\xi $ is orthogonal to $\mathcal{D}$ and therefore the tangent bundle splits orthogonally:
\begin{equation}
TM=\mathcal{D}\oplus \langle\xi\rangle.
\end{equation}

An analogue of the Kenmotsu manifold \cite{ke} in paracontact geometry will be further considered.

\begin{definition}\cite{olsz}
We say that the almost paracontact metric structure $(\varphi , \xi , \eta , g)$ is para-Kenmotsu if the Levi-Civita connection $\nabla$ of $g$ satisfies $(\nabla_X \varphi)Y=g(\varphi X, Y)\xi-\eta(Y)\varphi X$, for any $X$, $Y\in \mathfrak{X}(M)$.
\end{definition}

Note that the para-Kenmotsu structure was introduced by J. We{\l}yczko in \cite{we} for $3$-dimensional normal almost paracontact metric structures. A similar notion called $P$-Kenmotsu structure appears in the paper of B. B. Sinha and K. L. Sai Prasad \cite{si}.

\medskip

We shall further give some immediate properties of this structure.

\begin{proposition}\label{p1}
On a para-Kenmotsu manifold $(M,\varphi , \xi , \eta , g)$, the following relations hold:
\begin{equation}\nabla \xi=I_{\mathfrak{X}(M)}-\eta\otimes \xi\end{equation}
\begin{equation}\eta(\nabla_X\xi)=0, \ \ \nabla_{\xi}\xi=0,\end{equation}
\begin{equation}R(X,Y)\xi=\eta(X)Y-\eta(Y)X,\end{equation}
\begin{equation}\eta(R(X,Y)Z)=-\eta(X)g(Y,Z)+\eta(Y)g(X,Z), \ \ \eta(R(X,Y)\xi)=0,\end{equation}
\begin{equation}\nabla \eta=g-\eta\otimes\eta, \ \ \nabla_{\xi} \eta=0,\end{equation}
\begin{equation}L_{\xi}\varphi=0, \ \ L_{\xi}\eta=0, \ \ L_{\xi}(\eta\otimes \eta)=0, \ \ L_{\xi}g=2(g-\eta\otimes \eta),\end{equation}
where $R$ is the Riemann curvature tensor field and $\nabla$ is the Levi-Civita connection associated to $g$.
Moreover, $\eta$ is closed, the distribution $\mathcal{D}$ is involutive and the Nijenhuis tensor field of $\varphi$ vanishes identically.
\end{proposition}
\begin{proof}
Taking $Y:=\xi$ in $\nabla_X\varphi Y-\varphi(\nabla_XY)=g(\varphi X,Y)\xi-\eta(Y)\varphi X$ follows $\varphi (\nabla_X \xi)=\varphi X$ and applying $\varphi$ we obtain $\nabla_X\xi-\eta(\nabla_X\xi)\xi=X-\eta(X)\xi$. But $X(g(\xi,\xi))=2g(\nabla_X\xi,\xi)$ and so $\eta(\nabla_X\xi)=g(\nabla_X\xi,\xi)=0$. Therefore, $\nabla_X\xi=X-\eta(X)\xi$. In particular, $\nabla_{\xi}\xi=0$.

\medskip

Replacing now the expression of $\nabla \xi$ in $R(X,Y)\xi:=\nabla_X\nabla_Y\xi-\nabla_Y\nabla_X\xi-\nabla_{[X,Y]}\xi$, from a direct computation we get $R(X,Y)\xi=\eta(X)Y-\eta(Y)X$. Also $\eta(R(X,Y)Z)=g(R(X,Y)Z,\xi)=-g(R(X,Y)\xi,Z)=-[\eta(X)g(Y,Z)-\eta(Y)g(X,Z)]$. In particular, $\eta(R(X,Y)\xi)=0$.

\medskip

Compute $(\nabla_X \eta)Y:=X(\eta(Y))-\eta(\nabla_XY)=X(g(Y,\xi))-g(\nabla_XY,\xi)=g(Y,\nabla_X\xi)=g(X,Y)-\eta(X)\eta(Y)$. In particular, $(\nabla_{\xi} \eta)Y=0$.

\medskip

Express the Lie derivatives along $\xi$ as follows:
$$(L_{\xi}\varphi)(X):=[\xi,\varphi X]-\varphi([\xi,X])=\nabla_{\xi}\varphi X-\nabla_{\varphi X}\xi-\varphi(\nabla_{\xi}X)+\varphi (\nabla_X\xi)=$$$$=
\nabla_{\xi}\varphi X-\varphi(\nabla_{\xi}X):=(\nabla_{\xi}\varphi)X=0,$$
$$(L_{\xi}\eta)(X):=\xi(\eta(X))-\eta([\xi,X])=\xi(g(X,\xi))-g(\nabla_{\xi}X,\xi)+g(\nabla_X\xi,\xi)=$$$$=
g(X,\nabla_{\xi}\xi)+\eta(\nabla_X\xi)=0,$$
$$(L_{\xi}(\eta\otimes \eta))(X,Y):=\xi(\eta(X)\eta(Y))-\eta([\xi,X])\eta(Y)-\eta(X)\eta([\xi,Y])=$$$$=
\eta(X)\xi(\eta(Y))+\eta(Y)\xi(\eta(X))-\eta(\nabla_{\xi}X)\eta(Y)+\eta(\nabla_X{\xi})\eta(Y)
-\eta(X)\eta(\nabla_{\xi}Y)+\eta(X)\eta(\nabla_Y{\xi})=$$$$=\eta(X)[\xi(g(Y,\xi))-g(\nabla_{\xi}Y,\xi)]+
\eta(Y)[\xi(g(X,\xi))-g(\nabla_{\xi}X,\xi)]=$$$$=\eta(X)g(Y,\nabla_{\xi}\xi)-\eta(Y)g(X,\nabla_{\xi}\xi)=0$$
and
$$(L_{\xi}g)(X,Y):=\xi(g(X,Y))-g([\xi,X],Y)-g(X,[\xi,Y])=$$$$=\xi(g(X,Y))-g(\nabla_{\xi}X,Y)+g(\nabla_X\xi,Y)-
g(X,\nabla_{\xi}Y)+g(X,\nabla_Y\xi)=$$$$=g(\nabla_X\xi,Y)+g(X,\nabla_Y\xi)=2[g(X,Y)-\eta(X)\eta(Y)].$$

From $\nabla_X\xi=X-\eta(X)\xi$ and $\nabla_X\varphi Y-\varphi (\nabla_XY)=g(\varphi X,Y)\xi-\eta(Y)\varphi X$ we consequently obtain:
$$(d\eta)(X,Y):=X(\eta(Y))-Y(\eta(X))-\eta([X,Y])=X(g(Y,\xi))-Y(g(X,\xi))-g([X,Y],\xi)=$$$$
=X(g(Y,\xi))-g(\nabla_XY,\xi)-Y(g(X,\xi))+g(\nabla_YX,\xi)=g(Y,\nabla_X\xi)-g(X,\nabla_Y\xi)=0$$
and
$$N_{\varphi}(X,Y):=\varphi^2[X,Y]+[\varphi X,\varphi Y]-\varphi[\varphi X,Y]-\varphi[X,\varphi Y]=$$
$$=\varphi^2(\nabla_XY)-\varphi(\nabla_X\varphi Y)-\varphi^2(\nabla_YX)+\varphi(\nabla_Y\varphi X)+\nabla_{\varphi X}\varphi Y-\varphi(\nabla_{\varphi X}Y)-\nabla_{\varphi Y}\varphi X+
\varphi(\nabla_{\varphi Y}X)=$$$$=[g(\varphi^2 X,Y)-g(X,\varphi^2Y)]\xi=0.$$
\end{proof}

\bigskip

\begin{example}
Let $M=\mathbb{R}^3$ and $(x,y,z)$ be the standard coordinates in $\mathbb{R}^3$. Set
$$\varphi:=\frac{\partial}{\partial y}\otimes dx+ \frac{\partial}{\partial x}\otimes dy, \ \ \xi:=-\frac{\partial}{\partial z}, \ \ \eta:=-dz,$$
$$g:=dx\otimes dx-dy\otimes dy+dz\otimes dz.$$
Then $(\varphi , \xi , \eta , g)$ is a para-Kenmotsu structure on $\mathbb{R}^3$. Indeed, being sufficiently to verify the conditions in the definition on a linearly independent system of vector fields, consider it,
$$E_1:=\frac{\partial}{\partial x}, \ \ E_2:=\frac{\partial}{\partial y}, \ \ E_3:=-\frac{\partial}{\partial z}.$$
Follows
$$\varphi E_1=E_2, \ \ \varphi E_2=E_1, \ \ \varphi E_3=0,$$
$$\eta(E_1)=0, \ \ \eta(E_2)=0, \ \ \eta(E_3)=1,$$
$$[E_1,E_2]=0, \ \ [E_2,E_3]=0, \ \ [E_3,E_1]=0$$
and the Levi-Civita connection $\nabla$ is deduced from Koszul's formula
$$2g(\nabla_XY,Z)=X(g(Y,Z))+Y(g(Z,X))-Z(g(X,Y))-$$$$-g(X,[Y,Z])+g(Y,[Z,X])+g(Z,[X,Y]),$$
precisely,
$$\nabla_{E_1}E_1=-E_3, \ \ \nabla_{E_1}E_2=0, \ \ \nabla_{E_1}E_3=E_1,$$
$$\nabla_{E_2}E_1=0, \ \ \nabla_{E_2}E_2=E_3, \ \ \nabla_{E_2}E_3=E_2,$$
$$\nabla_{E_3}E_1=E_1, \ \ \nabla_{E_3}E_2=E_2, \ \ \nabla_{E_3}E_3=0.$$
\end{example}

\bigskip

In this setting, we shall study $\eta$-Ricci solitons for the cases: $(\xi,\cdot)_{R}\cdot S=0$, $(\xi,\cdot)_{S}\cdot R=0$, $(\xi,\cdot)_{W_2}\cdot S=0$ and $(\xi,\cdot)_{S}\cdot W_2=0$, where for $\mathcal{T}$ a $(1,3)$-type tensor, we denote by $\cdot$ the derivation of the tensor algebra at each point of the tangent space:
\begin{itemize}
  \item $((\xi,X)_{\mathcal{T}}\cdot S)(Y,Z):=((\xi\wedge_\mathcal{T}X)\cdot S)(Y,Z):=S((\xi\wedge_{\mathcal{T}}X)Y,Z)+S(Y,(\xi\wedge_{\mathcal{T}}X)Z)$, for $(X\wedge_{\mathcal{T}}Y)Z:=\mathcal{T}(X,Y)Z$;
  \item $((\xi,X)_{S}\cdot \mathcal{T})(Y,Z)W:=(\xi\wedge_SX)\mathcal{T}(Y,Z)W+\mathcal{T}((\xi\wedge_SX)Y,Z)W+\linebreak \mathcal{T}(Y,(\xi\wedge_SX)Z)W+\mathcal{T}(Y,Z)(\xi\wedge_SX)W$, for $(X\wedge_SY)Z:=S(Y,Z)X-S(X,Z)Y$.
\end{itemize}

\section{$\eta$-Ricci solitons on $(M,\varphi , \xi , \eta , g)$}

Let $(M,\varphi , \xi , \eta , g)$ be an almost paracontact metric manifold. Consider the equation
\begin{equation}\label{e8}
L_{\xi}g+2S+2\lambda g+2\mu\eta\otimes \eta=0,
\end{equation}
where $L_{\xi}$ is the Lie derivative operator along the vector field $\xi$, $S$ is the Ricci curvature tensor field of the metric $g$, and $\lambda$ and $\mu$ are real constants. Writing $L_{\xi}g$ in terms of the Levi-Civita connection $\nabla$, we obtain:
\begin{equation}\label{e9}
2S(X,Y)=
-g(\nabla_X\xi,Y)-g(X,\nabla_Y\xi)-2\lambda g(X,Y)-2\mu\eta(X)\eta(Y),
\end{equation}
for any $X$, $Y\in \mathfrak{X}(M)$.

The data $(g,\xi,\lambda,\mu)$ which satisfy the equation (\ref{e8}) is said to be an \textit{$\eta$-Ricci soliton} on $M$ \cite{ch}; in particular, if $\mu=0$, $(g,\xi,\lambda)$ is a \textit{Ricci soliton} \cite{ham} and it is called \textit{shrinking}, \textit{steady} or \textit{expanding} according as $\lambda$ is negative, zero or positive, respectively \cite{chlu}.

\bigskip

An important geometrical object in studying Ricci solitons is well-known to be a symmetric $(0,2)$-tensor field which is parallel with respect to the Levi-Civita connection, some of its geometrical properties being described in \cite{be}, \cite{cr1} etc. In the same manner as in \cite{cacr} we shall state the existence of $\eta$-Ricci solitons in our settings.

Consider now $\alpha$ such a symmetric $(0,2)$-tensor field which is parallel with respect to the Levi-Civita connection ($\nabla \alpha=0$). From the Ricci identity $\nabla^2\alpha (X,Y;Z,W)-\nabla^2\alpha (X,Y;W,Z)=0$, one obtains for any $X$, $Y$, $Z$, $W\in \mathfrak{X}(M)$ \cite{sh}
\begin{equation}
\alpha(R(X,Y)Z,W)+\alpha(Z,R(X,Y)W)=0.
\end{equation}
In particular, for $Z=W:=\xi$ from the symmetry of $\alpha$ follows $\alpha(R(X,Y)\xi,\xi)=0$, for any $X$, $Y\in \mathfrak{X}(M)$.

\bigskip

If $(\varphi , \xi , \eta , g)$ is a para-Kenmotsu structure on $M$, from Proposition \ref{p1} we have $R(X,Y)\xi=\eta(X)Y-\eta(Y)X$ and replacing this expression in $\alpha$ we get:
\begin{equation}\label{e22}
\alpha(Y,\xi)-\eta(Y)\alpha(\xi,\xi)=0,
\end{equation}
for any $Y\in \mathfrak{X}(M)$, equivalent to:
\begin{equation}\label{e20}
\alpha(Y,\xi)-\alpha(\xi,\xi)g(Y,\xi)=0,
\end{equation}
for any $Y\in \mathfrak{X}(M)$. Differentiating the equation (\ref{e20}) covariantly with respect to the vector field $X\in \mathfrak{X}(M)$ we obtain:
$$\alpha(\nabla_XY,\xi)+\alpha(Y,\nabla_X\xi)=\alpha(\xi,\xi)[g(\nabla_XY,\xi)+g(Y,\nabla_X\xi)]$$
and substituting the expression of $\nabla_X\xi=X-\eta(X)\xi$ we get:
\begin{equation}\label{e21}
\alpha(Y,X)=\alpha(\xi,\xi)g(Y,X),
\end{equation}
for any $X$, $Y\in \mathfrak{X}(M)$. As $\alpha$ is $\nabla$-parallel, follows $\alpha(\xi,\xi)$ is constant and we conclude:
\begin{proposition}
Under the hypotheses above, any parallel symmetric $(0,2)$-tensor field is a constant multiple of the metric.
\end{proposition}

Because on a para-Kenmotsu manifold $(M,\varphi, \xi, \eta,g)$, $\nabla_X \xi=X-\eta(X)\xi$ and \linebreak $L_{\xi}g=2(g-\eta\otimes \eta)$, the equation (\ref{e9}) becomes:
\begin{equation}\label{e1}
S(X,Y)=-(\lambda+1)g(X,Y)-(\mu-1)\eta(X)\eta(Y).
\end{equation}
In particular, $S(X,\xi)=S(\xi,X)=-(\lambda+\mu)\eta(X)$. But it is known \cite{s:z} that on a $(2n+1)$-dimensional paracontact manifold $M$, $S(X,\xi)=-(\dim (M)-1)\eta(X)=-2n\eta(X)$, so:
\begin{equation}
\lambda+\mu=2n.
\end{equation}

In this case, the Ricci operator $Q$ defined by $g(QX,Y):=S(X,Y)$ has the expression:
\begin{equation}\label{e3}
QX=-(2n+1-\mu)X-(\mu-1)\eta(X)\xi.
\end{equation}

\bigskip

Now we shall apply the previous results to $\eta$-Ricci solitons.

\begin{theorem}
Let $(M,\varphi,\xi,\eta,g)$ be a para-Kenmotsu manifold.
Assume that the symmetric $(0,2)$-tensor field $\alpha:=L_{\xi}g+2S+2\mu \eta\otimes \eta$ is parallel with respect to the Levi-Civita connection associated to $g$. Then $(g,\xi,\mu)$ yields an $\eta$-Ricci soliton.
\end{theorem}
\begin{proof}
Compute
$$\alpha(\xi,\xi)=(L_{\xi}g)(\xi,\xi)+2S(\xi,\xi)+2\mu \eta(\xi)\eta(\xi)=-2\lambda,$$
so $\lambda=-\frac{\displaystyle 1}{\displaystyle 2}\alpha(\xi,\xi)$.
From (\ref{e21}) we get $\alpha(X,Y)=-2\lambda g(X,Y)$, for any $X$, $Y\in \mathfrak{X}(M)$. Therefore,
$L_{\xi}g+2S+2\mu \eta\otimes \eta=-2\lambda g$.
\end{proof}

For $\mu=0$ follows $L_{\xi}g+2S+4ng=0$ and we conclude:

\begin{corollary}
On a para-Kenmotsu manifold $(M,\varphi,\xi,\eta,g)$ with the property that the symmetric $(0,2)$-tensor field $\alpha:=L_{\xi}g+2S$ is parallel with respect to the Levi-Civita connection associated to $g$, the relation (\ref{e8}), for $\mu=0$ and $\lambda=2n$, defines a Ricci soliton on $M$.
\end{corollary}

Conversely, we shall study the consequences of the existence of $\eta$-Ricci solitons on a para-Kenmotsu manifold. From (\ref{e1}) we deduce:

\begin{proposition}
If (\ref{e8}) defines an $\eta$-Ricci soliton on the para-Kenmotsu manifold \linebreak $(M, \varphi , \xi , \eta , g)$, then $(M,g)$ is quasi-Einstein.
\end{proposition}
Recall that the manifold is called \textit{quasi-Einstein} if the Ricci curvature tensor field $S$ is a linear combination (with real scalars $\lambda$ and $\mu$ respectively, with $\mu\neq 0$) of $g$ and the tensor product of a non-zero $1$-form $\eta$ satisfying $\eta(X)=g(X,\xi)$, for $\xi$ a unit vector field \cite{mai} and respectively, \textit{Einstein} if $S$ is collinear with $g$.

\begin{proposition}
If $(\varphi , \xi , \eta , g)$ is a para-Kenmotsu structure on $M$ and (\ref{e8}) defines an $\eta$-Ricci soliton on $M$, then:
\begin{enumerate}
  \item $Q\circ \varphi=\varphi \circ Q$;
  \item $Q$ and $S$ are parallel along $\xi$.
\end{enumerate}
\end{proposition}
\begin{proof}
The first statement follows from a direct computation and for the second one, note that $(\nabla_{\xi}Q)X:=\nabla_{\xi}QX-Q(\nabla_{\xi}X)$ and $(\nabla_{\xi}S)(X,Y):=\xi(S(X,Y))-S(\nabla_{\xi}X,Y)-S(X,\nabla_{\xi}Y)$ and replace $Q$ and $S$ from (\ref{e3}) and (\ref{e1}).
\end{proof}

A particular case arise when the manifold is \textit{$\varphi$-Ricci symmetric}, which means that $\varphi^2 \circ \nabla Q=0$, fact stated in the next proposition:
\begin{proposition}
Let $(M,\varphi , \xi , \eta , g)$ be a para-Kenmotsu manifold. If $M$ is $\varphi$-Ricci symmetric and (\ref{e8}) defines an $\eta$-Ricci soliton on $M$, then $\mu=1$, $\lambda=2n-1$ and $(M,g)$ is Einstein manifold.
\end{proposition}
\begin{proof}
Replacing $Q$ from (\ref{e3}) in $(\nabla_XQ)Y:=\nabla_XQY-Q(\nabla_XY)$ and applying $\varphi^2$ we obtain:
$$(\mu-1)\eta(Y)[X-\eta(X)\xi]=0,$$
for any $X$, $Y\in \mathfrak{X}(M)$. Follows $\mu=1$, $\lambda=2n-1$ and $S=-2ng$.
\end{proof}

In particular, the existence of an $\eta$-Ricci soliton on a para-Kenmotsu manifold which is \textit{Ricci symmetric} (i.e. $\nabla S=0$) implies that $(M,g)$ is Einstein manifold. Remark that the class of Ricci symmetric manifolds represents an extension of the class of Einstein manifolds to which belong also the locally symmetric manifolds (i.e. those satisfying $\nabla R=0$). The condition $\nabla S=0$ implies $R\cdot S=0$ and the manifolds satisfying this condition are called Ricci semisymmetric \cite{n}.

\bigskip

We end these considerations by giving an example of $\eta$-Ricci soliton on a para-Kenmotsu manifold.

\bigskip

\begin{example}
Let $M=\mathbb{R}^3$ and $(x,y,z)$ be the standard coordinates in $\mathbb{R}^3$. Set
$$\varphi:=\frac{\partial}{\partial y}\otimes dx+\frac{\partial}{\partial x}\otimes dy, \ \ \xi:=-\frac{\partial}{\partial z}, \ \ \eta:=-dz,$$
$$g:=dx\otimes dx-dy\otimes dy+dz\otimes dz$$
and consider the linearly independent system of vector fields
$$E_1:=\frac{\partial}{\partial x}, \ \ E_2:=\frac{\partial}{\partial y}, \ \ E_3:=-\frac{\partial}{\partial z}.$$
Follows
$$\nabla_{E_1}E_1=-E_3, \ \ \nabla_{E_1}E_2=0, \ \ \nabla_{E_1}E_3=E_1, \ \ \nabla_{E_2}E_1=0, \ \ \nabla_{E_2}E_2=E_3,$$$$\nabla_{E_2}E_3=E_2, \ \ \nabla_{E_3}E_1=E_1, \ \ \nabla_{E_3}E_2=E_2, \ \ \nabla_{E_3}E_3=0.$$

Then the Riemann and the Ricci curvature tensor fields are given by:
$$R(E_1,E_2)E_2=E_1, \ \ R(E_1,E_3)E_3=-E_1, \ \ R(E_2,E_1)E_1=-E_2,$$ $$R(E_2,E_3)E_3=-E_2, \ \ R(E_3,E_1)E_1=E_3, \ \ R(E_3,E_2)E_2=-E_3,$$
$$S(E_1,E_1)=0, \ \ S(E_2,E_2)=0, \ \ S(E_3,E_3)=-2.$$

In this case, from (\ref{e1}), for $\lambda=-1$ and $\mu=3$, the data $(g,\xi,\lambda,\mu)$ is an $\eta$-Ricci soliton on $(\mathbb{R}^3, \varphi , \xi , \eta , g)$.
\end{example}

\bigskip

In what follows we shall consider $\eta$-Ricci solitons requiring for the curvature to satisfy $(\xi,\cdot)_{R}\cdot S=0$, $(\xi,\cdot)_{S}\cdot R=0$, $(\xi,\cdot)_{W_2}\cdot S=0$ and $(\xi,\cdot)_{S}\cdot W_2=0$, respectively, where the $W_2$-curvature tensor field is the curvature tensor introduced by G. P Pokhariyal and R. S. Mishra in \cite{po}:
\begin{equation}
W_2(X,Y)Z:=R(X,Y)Z+\frac{1}{\dim (M)-1}[g(X,Z)QY-g(Y,Z)QX]=$$$$=R(X,Y)Z+\frac{1}{2n}[g(X,Z)QY-g(Y,Z)QX].
\end{equation}

\subsection{$\eta$-Ricci solitons on para-Kenmotsu manifolds \\
satisfying $(\xi,\cdot)_{R}\cdot S=0$}

The condition that must be satisfied by $S$ is:
\begin{equation}
S(R(\xi,X)Y,Z)+S(Y,R(\xi,X)Z)=0,
\end{equation}
for any $X$, $Y$, $Z\in \mathfrak{X}(M)$.

\bigskip

Replacing the expression of $S$ from (\ref{e1}) and from the symmetries of $R$ we get:
\begin{equation}
(\mu-1)[\eta(Y)g(X,Z)+\eta(Z)g(X,Y)-2\eta(X)\eta(Y)\eta(Z)]=0,
\end{equation}
for any $X$, $Y$, $Z\in \mathfrak{X}(M)$.

For $Z:=\xi$ we have:
\begin{equation}
(\mu-1)g(\varphi X,\varphi Y)=0,
\end{equation}
for any $X$, $Y\in \mathfrak{X}(M)$. But $\lambda+\mu=2n$ and we can state:

\begin{theorem}\label{p}
If $(\varphi,\xi,\eta,g)$ is a para-Kenmotsu structure on the $(2n+1)$-dimensional manifold $M$, $(g,\xi,\lambda,\mu)$ is an $\eta$-Ricci soliton on $M$ and $(\xi,\cdot)_{R}\cdot S=0$, then $\mu=1$ and $\lambda=2n-1$, so $(M,g)$ is Einstein manifold.
\end{theorem}

\begin{corollary}
On a para-Kenmotsu manifold $(M,\varphi,\xi,\eta,g)$ satisfying $(\xi,\cdot)_{R}\cdot S=0$, there is no Ricci soliton with the potential vector field $\xi$.
\end{corollary}

\subsection{$\eta$-Ricci solitons on para-Kenmotsu manifolds \\
satisfying $(\xi,\cdot)_{S}\cdot R=0$}

The condition that must be satisfied by $S$ is:

\begin{equation}\label{e777}
S(X,R(Y,Z)W)\xi-S(\xi,R(Y,Z)W)X+S(X,Y)R(\xi,Z)W-$$
$$-S(\xi,Y)R(X,Z)W+S(X,Z)R(Y,\xi)W-S(\xi,Z)R(Y,X)W+$$
$$+S(X,W)R(Y,Z)\xi-S(\xi,W)R(Y,Z)X=0,
\end{equation}
for any $X$, $Y$, $Z$, $W\in \mathfrak{X}(M)$.

\bigskip

Taking the inner product with $\xi$, the relation (\ref{e777}) becomes:
\begin{equation}
S(X,R(Y,Z)W)-S(\xi,R(Y,Z)W)\eta(X)+$$$$+S(X,Y)\eta(R(\xi,Z)W)-S(\xi,Y)\eta(R(X,Z)W)+
S(X,Z)\eta(R(Y,\xi)W)-$$$$-S(\xi,Z)\eta(R(Y,X)W)+S(X,W)\eta(R(Y,Z)\xi)-S(\xi,W)\eta(R(Y,Z)X)=0,
\end{equation}
for any $X$, $Y$, $Z$, $W\in \mathfrak{X}(M)$.

Replacing the expression of $S$ from (\ref{e1}), we get:
\begin{equation}
(\lambda+1)[g(X,R(Y,Z)W)-2\eta(X)\eta(Z)g(Y,W)+2\eta(X)\eta(Y)g(Z,W)-$$$$-g(X,Y)g(Z,W)+g(X,Z)g(Y,W)]+$$$$+
(\mu-1)[\eta(Y)\eta(W)g(X,Z)-\eta(Z)\eta(W)g(X,Y)]=0,
\end{equation}
for any $X$, $Y$, $Z$, $W\in \mathfrak{X}(M)$.

For $W:=\xi$ we have:
\begin{equation}
(2\lambda+\mu+1)[\eta(Y)g(X,Z)-\eta(Z)g(X,Y)]=0,
\end{equation}
for any $X$, $Y$, $Z\in \mathfrak{X}(M)$, which is equivalent to
\begin{equation}
(2\lambda+\mu+1)g(X,R(Y,Z)\xi)=0,
\end{equation}
for any $X$, $Y$, $Z\in \mathfrak{X}(M)$. But $\lambda+\mu=2n$, so $4n+1-\mu=0$ and we can state:

\begin{theorem}
If $(\varphi,\xi,\eta,g)$ is a para-Kenmotsu structure on the $(2n+1)$-dimensional manifold $M$, $(g,\xi,\lambda,\mu)$ is an $\eta$-Ricci soliton on $M$ and $(\xi,\cdot)_{S}\cdot R=0$, then $\mu=4n+1$ and $\lambda=-2n-1$.
\end{theorem}

\begin{corollary}
On a para-Kenmotsu manifold $(M,\varphi, \xi,\eta,g)$ satisfying \linebreak $(\xi,\cdot)_{S}\cdot R=0$, there is no Ricci soliton with the potential vector field $\xi$.
\end{corollary}

\subsection{$\eta$-Ricci solitons on para-Kenmotsu manifolds \\
satisfying $(\xi,\cdot)_{W_2}\cdot S=0$}

The condition that must be satisfied by $S$ is:
\begin{equation}
S(W_2(\xi,X)Y,Z)+S(Y,W_2(\xi,X)Z)=0,
\end{equation}
for any $X$, $Y$, $Z\in \mathfrak{X}(M)$.

\bigskip

Replacing the expression of $S$ from (\ref{e1}) we get:
\begin{equation}
\frac{(\mu-1)(2\lambda+\mu+1-2n)}{2n}[\eta(Y)g(X,Z)+\eta(Z)g(X,Y)-2\eta(X)\eta(Y)\eta(Z)]=0,
\end{equation}
for any $X$, $Y$, $Z\in \mathfrak{X}(M)$.

For $Z:=\xi$ we have:
\begin{equation}
(\mu-1)(2\lambda+\mu+1-2n)g(\varphi X,\varphi Y)=0,
\end{equation}
for any $X$, $Y\in \mathfrak{X}(M)$. But $\lambda+\mu=2n$, so $(\mu-1)(2n+1-\mu)=0$ and we can state:

\begin{theorem}\label{pro1}
If $(\varphi,\xi,\eta,g)$ is a para-Kenmotsu structure on the $(2n+1)$-dimensional manifold $M$, $(g,\xi,\lambda,\mu)$ is an $\eta$-Ricci soliton on $M$ and $(\xi,\cdot)_{W_2}\cdot S=0$, then $\mu=1$ and $\lambda=2n-1$ or $\mu=2n+1$ and $\lambda=-1$.
\end{theorem}

\begin{corollary}
On a para-Kenmotsu manifold $(M,\varphi, \xi,\eta,g)$ satisfying \linebreak $(\xi,\cdot)_{W_2}\cdot S=0$, there is no Ricci soliton with the potential vector field $\xi$.
\end{corollary}

\subsection{$\eta$-Ricci solitons on para-Kenmotsu manifolds \\
satisfying $(\xi,\cdot)_{S}\cdot W_2=0$}

The condition that must be satisfied by $S$ is:
\begin{equation}\label{e7}
S(X,W_2(Y,Z)V)\xi-S(\xi,W_2(Y,Z)V)X+S(X,Y)W_2(\xi,Z)V-$$
$$-S(\xi,Y)W_2(X,Z)V+S(X,Z)W_2(Y,\xi)V-S(\xi,Z)W_2(Y,X)V+$$$$+S(X,V)W_2(Y,Z)\xi-S(\xi,V)W_2(Y,Z)X=0,
\end{equation}
for any $X$, $Y$, $Z$, $V\in \mathfrak{X}(M)$.

Taking the inner product with $\xi$, the relation (\ref{e7}) becomes:
\begin{equation}
S(X,W_2(Y,Z)V)-S(\xi,W_2(Y,Z)V)\eta(X)+$$$$+S(X,Y)\eta(W_2(\xi,Z)V)-S(\xi,Y)\eta(W_2(X,Z)V)+
S(X,Z)\eta(W_2(Y,\xi)V)-$$$$-S(\xi,Z)\eta(W_2(Y,X)V)+S(X,V)\eta(W_2(Y,Z)\xi)-S(\xi,V)\eta(W_2(Y,Z)X)=0,
\end{equation}
for any $X$, $Y$, $Z$, $V\in \mathfrak{X}(M)$.

Replacing the expression of $S$ from (\ref{e1}), we get:
\begin{equation}
(\lambda+1)[g(X,R(Y,Z)V)-\frac{2\lambda+\mu+1-2n}{2n}(g(X,Z)g(Y,V)-g(X,Y)g(Z,V))+$$$$+\frac{2\lambda+\mu+1-4n}{2n}
(\eta(X)\eta(Z)g(Y,V)-\eta(X)\eta(Y)g(Z,V))+$$$$+\frac{(\mu-1)(\lambda+\mu-2n)}{2n}
(\eta(Z)\eta(V)g(X,Y)-\eta(Y)\eta(V)g(X,Z))=0,
\end{equation}
for any $X$, $Y$, $Z$, $V\in \mathfrak{X}(M)$.

For $V:=\xi$ we have:
\begin{equation}
[(\lambda+1)^2+(\lambda+\mu)^2-2n(2\lambda+\mu+1)][\eta(Y)g(X,Z)-\eta(Z)g(X,Y)]=0,
\end{equation}
for any $X$, $Y$, $Z\in \mathfrak{X}(M)$, which is equivalent to
\begin{equation}
[(\lambda+1)^2+(\lambda+\mu)^2-2n(2\lambda+\mu+1)]g(X,R(Y,Z)\xi)=0,
\end{equation}
for any $X$, $Y$, $Z\in \mathfrak{X}(M)$. But $\lambda+\mu=2n$, so $\mu^2-2(n+1)\mu+2n+1=0$ and we can state:

\begin{theorem}
If $(\varphi,\xi,\eta,g)$ is a para-Kenmotsu structure on the $(2n+1)$-dimensional manifold $M$, $(g,\xi,\lambda,\mu)$ is an $\eta$-Ricci soliton on $M$ and $(\xi,\cdot)_{S}\cdot W_2=0$, then $\mu=1$ and $\lambda=2n-1$ or $\mu=2n+1$ and $\lambda=-1$.
\end{theorem}

\begin{corollary}
On a para-Kenmotsu manifold $(M,\varphi, \xi,\eta,g)$ satisfying \linebreak $(\xi,\cdot)_{S}\cdot W_2=0$, there is no Ricci soliton with the potential vector field $\xi$.
\end{corollary}

{\bf Acknowledgement.} The author acknowledges the support by the research grant PN-II-ID-PCE-2011-3-0921.

\bigskip

\textit{Department of Mathematics}

\textit{West University of Timi\c{s}oara}

\textit{Bld. V. P\^{a}rvan nr. 4, 300223, Timi\c{s}oara, Rom\^{a}nia}

\textit{adarablaga@yahoo.com}
\end{document}